\documentclass{amsart}

\theoremstyle{definition}
\newtheorem{theorem}{Theorem}[section]

\newtheorem{corollary}[theorem]{Corollary}
\newtheorem{note}[theorem]{Note}
\newtheorem{prop}[theorem]{Proposition}

\theoremstyle{remark}
\newtheorem{remark}[theorem]{Remark}

\numberwithin{equation}{section}

\reversemarginpar

\newcommand{\realpart}{\mathop{\rm Re}\nolimits}

\newcommand{\ba}{\begin{eqnarray}}
\newcommand{\ea}{\end{eqnarray}}

\begin{document}

\title{A probabilistic approach to some binomial identities}

\author[]{Christophe Vignat}
\address{Information Theory Laboratory, E.P.F.L., 1015 Lausanne, Switzerland}
\email{christophe.vignat@epfl.ch}

\author[]{Victor H. Moll}
\address{Department of Mathematics,
Tulane University, New Orleans, LA 70118}
\email{vhm@math.tulane.edu}


\date{\today}


\begin{abstract}
Classical binomial identities are established by giving probabilistic 
interpretations to the summands. The examples include Vandermonde 
identity and some generalizations.
\end{abstract}

\maketitle


\vskip 20pt 

\section{Introduction} 
\label{S:intro} 

The evaluation of finite sums involving binomial coefficients appears 
throughout the undergraduate curriculum. Students are often exposed to
the identity 
\begin{equation}
\sum_{k=0}^{n} \binom{n}{k} = 2^{n}.
\label{two-bin}
\end{equation}
\noindent 
Elementary proofs abound: simply choose $x=y=1$ in the 
binomial expansion of $(x+y)^{n}$. 
The reader is surely aware of many other proofs, 
including some combinatorial in nature.

At the end of the previous century, the evaluation of these sums was 
trivialized by 
the work of H. Wilf and D. Zeilberger \cite{aequalsb}. In the preface to
the charming book \cite{aequalsb}, the authors begin with the phrase
\begin{center}
\texttt{You've been up all night working on your new theory, you found the 
answer, and it is in the form that involves factorials, binomial coefficients, 
and so on, ...} 
\end{center}
\noindent
and then proceed to introduce the method of {\em creative 
telescoping}. This technique provides
an automatic tool for the verification of this type of identities. 

Even in the presence of a powerful technique, such as the WZ-method, it is 
often a good pedagogical idea to present a 
simple identity from many different points of view. The reader will 
find in \cite{amdeberhan-2012a} this approach with the example 
\begin{equation}
\sum_{k=0}^{m} 2^{-2k} \binom{2k}{k} \binom{2m-k}{m} = 
\sum_{k=0}^{m} 2^{-2k} \binom{2k}{k} \binom{2m+1}{2k}.
\label{pretty-1}
\end{equation}

The current paper presents 
probabilistic arguments for the evaluation of certain 
binomial sums. The background required is minimal. The 
continuous random variables
$X$ considered here have a probability density 
function. This is a nonnegative function 
$f_{X}(x)$, such that
\begin{equation}
\Pr(X< x) = \int_{-\infty}^{x} f_{X}(y) \, dy.
\end{equation}
\noindent
In particular, $f_{X}$ must have total mass $1$.
Thus, all computations are reduced to the evaluation of integrals. For 
instance, the expectation of a function of the random variable $X$ is 
computed as 
\begin{equation}
\mathbb{E} g(X) = \int_{-\infty}^{\infty} g(y) f_{X}(y) \, dy.
\end{equation}
\noindent
In elementary courses, the reader has been exposed to normal 
random variables, written as 
$X \sim N(0,1)$, with density
\begin{equation}
f_{X}(x) = \frac{1}{\sqrt{2 \pi}} e^{-x^{2}/2}, 
\end{equation}
and exponential random variables, with probability density function
\begin{equation}
f(x;\lambda) = \begin{cases}
   \lambda e^{- \lambda x} & \text{ for } x \geq 0; \\
   0  & \text{ otherwise.}
\end{cases}
\end{equation}

The examples employed in the arguments presented here have 
a gamma distribution with shape parameter $k$ and scale parameter $\theta$, 
written as $X \sim \Gamma(k, \theta)$. These 
are defined by the density function
\begin{equation}
f(x;k, \theta) = 
\begin{cases}
x^{k-1} e^{-x/\theta}/\theta^{k} \Gamma(k), & \quad 
\text{ for } x \geq 0; \\
0 & \quad \text{ otherwise}.
\end{cases}
\end{equation}
\noindent
Here $\Gamma(s)$ is the classical gamma function, defined by  
\begin{equation}
\Gamma(s) = \int_{0}^{\infty} x^{s-1}e^{-x} \, dx
\end{equation}
\noindent
for $\realpart{s} > 0$. Observe that if $X \sim \Gamma(a,\theta)$, then 
$X = \theta Y$ where $Y \sim \Gamma(a,1)$. Moreover $\mathbb{E} X^{n} = 
\theta^{n} (a)_{n}$, where 
\begin{equation}
(a)_{n} = \frac{\Gamma(a+n)}{\Gamma(a)} = a(a+1) \cdots (a+n-1)
\end{equation}
\noindent 
is the Pochhammer symbol.  The 
main property of these random variables employed in this paper is the
following: assume 
$X_{i} \sim \Gamma(k_{i},\theta)$ are independent, then
\begin{equation}
X_{1} + \cdots + X_{n}  \sim \Gamma(k_{1} + \cdots + k_{n}, \theta).
\end{equation}
This follows from the fact that 
that the density probability function for the sum of two independent 
random variables is the convolution of the individual ones. 

Related random variables include those with a beta distribution
\begin{equation}
f_{a,b}(x) = 
\begin{cases}
x^{a-1}(1-x)^{b-1}/B(a,b) & \quad 
\text{ for } 0 \leq x \leq 1; \\
0  & \quad \text{ otherwise}.
\end{cases}
\end{equation}
\noindent
Here $B(a,b)$ is the beta function defined by 
\begin{equation}
B(a,b) = \int_{0}^{1} x^{a-1}(1-x)^{b-1} \, dx
\end{equation}
\noindent
and also the symmetric beta distributed 
random variable $Z_{c}$, with density proportional to 
$(1-x^{2})^{c-1}$ for $-1 \leq x \leq 1.$  The first class of random 
variables can
be generated as 
\begin{equation}
B_{a,b} = \frac{\Gamma_{a}}{\Gamma_{a} + \Gamma_{b}}, 
\label{fun-1}
\end{equation}
\noindent
where $\Gamma_{a}$ and $\Gamma_{b}$ are independent gamma distributed with 
shape parameters $a$ and $b$, respectively and the second type is 
distributed as $1 - 2B_{c,c}$, that is, 
\begin{equation}
Z_{c} = 1 - \frac{2 \Gamma_{c}}{\Gamma_{c} + \Gamma'_{c}} = \frac{\Gamma_{c} - \Gamma'_{c}}{\Gamma_{c} + 
\Gamma'_{c}},
\label{fun-2}
\end{equation}
\noindent
where $\Gamma_{c}$ and $\Gamma'_{c}$ are independent gamma distributed  with 
shape parameter $c$.  A well-known result is that $B_{a,b}$ and 
$\Gamma_{a}+\Gamma_{b}$ are independent in \eqref{fun-1}; similarly, 
$\Gamma_{c} + \Gamma'_{c}$ and $Z_{c}$ are independent in \eqref{fun-2}.

\section{A sum involving central binomial coefficients} 
\label{S:bincoeff}

Many finite sums may be evaluated via the generating function of terms 
appearing in them. For instance, a sum of the form
\begin{equation}
S_{2}(n) = \sum_{i+j=n} a_{i}a_{j}
\end{equation}
\noindent
is recognized as the coefficient of $x^{n}$ in the expansion of $f(x)^{2}$, 
where 
\begin{equation}
f(x) = \sum_{j=0}^{\infty} a_{j}x^{j}
\end{equation}
\noindent
is the generating function of the sequence $\{ a_{i} \}$. Similarly, 
\begin{equation}
S_{m}(n) = \sum_{k_{1} + \cdots + k_{m} =n} a_{k_{1}} \cdots a_{k_{m}} 
\end{equation}
\noindent
is given by the coefficient of $x^{n}$ in $f(x)^{m}$. The classical example
\begin{equation}
\frac{1}{\sqrt{1-4x}} = \sum_{j=0}^{\infty} \binom{2j}{j} x^{j}
\label{bin-exp1}
\end{equation}
\noindent 
gives the sums
\begin{equation}
\sum_{i=0}^{n} \binom{2i}{i} \binom{2n-2i}{n-i} = 4^{n}
\label{mult-identity-2}
\end{equation}
\noindent
and 
\begin{equation}
\label{mult-identity}
\sum_{k_{1}+\cdots + k_{m} = n} 
\binom{2k_{1}}{k_{1}} \cdots \binom{2k_{m}}{k_{m}} = 
\frac{2^{2n}}{n!} \frac{\Gamma( \tfrac{m}{2} + n )}{\Gamma(\tfrac{m}{2})}.
\end{equation}
\noindent
The powers of $(1-4x)^{-1/2}$ are obtained from the binomial expansion
\begin{equation}
(1 - 4x)^{-a} = \sum_{j=0}^{\infty} \frac{(a)_{j}}{j!}(4x)^{j},
\end{equation}
\noindent
where $(a)_{j}$ is the Pochhammer symbol.

\smallskip

The identity \eqref{mult-identity-2} is elementary and there are
many proofs in the 
literature. A nice combinatorial proof of \eqref{mult-identity}
appeared in $2006$ in this journal \cite{valerio-2006a}. In a 
more recent contribution, G. Chang and 
C. Xu \cite{chang-xu-2011a} present a 
probabilistic proof of these  identities. Their approach is elementary: take 
$m$ independent Gamma random variables $X_{i} \sim 
\Gamma(\tfrac{1}{2},1)$ and write 
\begin{equation}
\mathbb{E} \left( \sum_{i=1}^{m} X_{i} \right)^{n} = 
\sum_{k_{1}+\cdots+k_{m} = n} \binom{n}{k_{1}, \cdots, k_{m}} 
\mathbb{E}X_{1}^{k_{1}} \cdots 
\mathbb{E}X_{m}^{k_{m}}
\label{identity-0}
\end{equation}
where $\mathbb{E}$ denotes the expectation operator.
For each random variable $X_{i}$, the moments are given by
\begin{equation}
\mathbb{E} X_{i}^{k_{i}} = \frac{\Gamma( k_{i} + \tfrac{1}{2})}{\Gamma( 
\tfrac{1}{2} )} = 2^{-2k_{i}} \frac{(2k_{i})!}{k_{i}!} = 
\frac{k_{i}!}{2^{2k_{i}}} \binom{2k_{i}}{k_{i}},
\label{moments}
\end{equation}
\noindent
using Euler's duplication formula for the gamma function 
\begin{equation}
\Gamma(2z) = \frac{1}{\sqrt{\pi}} 2^{2z-1} \Gamma(z) \Gamma(z+ \tfrac{1}{2})
\end{equation}
(see \cite{nist}, $5.5.5$) to obtain the second form. The expression 
\begin{equation}
\binom{n}{k_{1}, \cdots, k_{m}} = \frac{n!}{k_{1}! \, k_{2}! \, \cdots \, 
k_{m}!}
\end{equation}
\noindent
for the multinomial coefficients shows that the right-hand side of 
\eqref{identity-0} is 
\begin{equation}
\frac{n!}{2^{2n}} \sum_{k_{1} + \cdots + k_{m}=n} 
\binom{2k_{1}}{k_{1}} \cdots \binom{2k_{m}}{k_{m}}.
\end{equation}
\noindent 
To evaluate the left-hand side of \eqref{identity-0}, recall that the sum of
$m$ independent $\Gamma \left(\tfrac{1}{2},1 \right)$ 
has a distribution of $\Gamma(\tfrac{m}{2},1)$. 
Therefore, the left-hand side of \eqref{identity-0} is 
\begin{equation}
\frac{\Gamma( \tfrac{m}{2} + n )}{\Gamma( \tfrac{m}{2} )}.
\end{equation}
This gives \eqref{mult-identity}. The special case $m=2$ produces 
\eqref{mult-identity-2}.

\section{More sums involving central binomial coefficients} 
\label{S:second} 

The next example deals with the identity 
\begin{equation}
\sum_{k=0}^{n} \binom{4k}{2k} \binom{4n-4k}{2n-2k} = 2^{4n-1} + 
2^{2n-1} \binom{2n}{n}
\label{sum-000}
\end{equation}
\noindent
that appears as entry $4.2.5.74$ in \cite{brychkov}.  The proof presented here 
employs the famous dissection technique, first 
introduced by Simpson  
\cite{simpson-1759} in the simplification of 
\begin{equation}
\frac{1}{2} \left( \mathbb{E}(X_{1}+X_{2})^{2n} + 
\mathbb{E}(X_{1}-X_{2})^{2n} \right),
\end{equation}
\noindent
where $X_{1}, \, X_{2}$ are independent random variables distributed 
as $\Gamma \left( \tfrac{1}{2}, 1 \right)$. 

The left-hand side is evaluated by expanding the binomials to obtain
\begin{multline}
\frac{1}{2} ( \mathbb{E}(X_{1}+X_{2})^{2n} + 
\mathbb{E}(X_{1}-X_{2})^{2n} )     =  \\  
\frac{1}{2}\sum_{k=0}^{2n} \binom{2n}{k} \mathbb{E} X_{1}^{k}  \,
\mathbb{E} X_{2}^{2n-k} +  
\frac{1}{2} \sum_{k=0}^{2n} (-1)^{k} \binom{2n}{k} \mathbb{E} X_{1}^{k} \, 
\mathbb{E} X_{2}^{2n-k} \nonumber 
\end{multline}
\noindent
This gives 
\begin{equation}
\frac{1}{2} ( \mathbb{E}(X_{1}+X_{2})^{2n} + 
\mathbb{E}(X_{1}-X_{2})^{2n} ) 
=    \sum_{k=0}^{n} \binom{2n}{2k} \mathbb{E} X_{1}^{2k}  \, 
\mathbb{E} X_{2}^{2n-2k}. \nonumber 
\label{nice-sum1}
\end{equation}
Using \eqref{moments}, this reduces to 
\begin{equation}
\frac{1}{2} \left( \mathbb{E}(X_{1}+X_{2})^{2n} + 
\mathbb{E}(X_{1}-X_{2})^{2n} \right)
= \frac{(2n)!}{2^{4n}} \sum_{k=0}^{n} \binom{4k}{2k} \binom{4n-4k}{2n-2k}.
\label{nice-sum1a}
\end{equation}

The random variable $X_{1}+X_{2}$ is $\Gamma(1,1)$ distributed, so 
\begin{equation}
\mathbb{E} (X_{1}+X_{2})^{2n}  = (2n)!,
\end{equation}
\noindent
and the random variable $X_{1}-X_{2}$ is distributed as 
$(X_{1}+X_{2})Z_{1/2}$, where $Z_{1/2}$ is independent of $X_{1}+X_{2}$
and has a symmetric beta distribution with density
$f_{Z_{1/2}}(z) = 1/\pi \, \sqrt{1-z^{2}}$. In particular, the even moments are
given by 
\begin{equation}
\label{even-mom}
\frac{1}{\pi} \int_{-1}^{1} \frac{z^{2n} \, dz}{\sqrt{1-z^{2}}} = 
\frac{1}{2^{2n}} \binom{2n}{n}.
\end{equation}
\noindent
Therefore, 
\begin{equation}
\mathbb{E} (X_{1}-X_{2})^{2n} = 
\mathbb{E} (X_{1}+X_{2})^{2n} \, \mathbb{E} Z^{2n}_{1/2} = 
\frac{(2n)!}{2^{2n}} \binom{2n}{n}.
\end{equation}
\noindent
It follows that 
\begin{equation}
\mathbb{E} (X_{1}+X_{2})^{2n} + 
\mathbb{E} (X_{1}-X_{2})^{2n}  = (2n)! + \frac{(2n)!}{2^{2n}} \binom{2n}{n}.
\label{nice-sum2}
\end{equation}

The evaluations \eqref{nice-sum1} and \eqref{nice-sum2} imply \eqref{sum-000}.

\section{An extension related to Legendre polynomials} 
\label{S:extension} 

A key point in the evaluation given in the previous section
is the elementary identity
\begin{equation}
\label{reduc-1}
1 + (-1)^{k} = \begin{cases}
       2 & \text{ if } k \text{ is even}; \\
       0 & \text{ otherwise. }
\end{cases}
\end{equation}
\noindent
This reduces the number of terms in the sum \eqref{nice-sum1} from $2n$ to $n$. 
A similar cancellation occurs for any $p \in \mathbb{N}$. Indeed, the natural 
extension of \eqref{reduc-1} is given by 
\begin{equation}
\label{reduc-2}
\sum_{j=0}^{p-1} \omega^{jr} = \begin{cases}
p & \text{ if } r \equiv 0 \pmod p; \\
0 & \text{ otherwise};
\end{cases}
\end{equation}
\noindent
Here $\omega = e^{2 \pi i /p}$ is a complex $p$-th root of unity. 
Observe that \eqref{reduc-2} reduces to \eqref{reduc-1} when $p=2$.

The goal of this section is to discuss the extension of \eqref{sum-000}. The 
main result is given in the next theorem. The 
Legendre polynomials appearing in the next theorem are defined by 
\begin{equation}
P_{n}(x) = \frac{1}{2^{n} \, n!} \left( \frac{d}{dx} \right)^{n} 
(x^{2}-1)^{n}.
\label{legen-def}
\end{equation}

\begin{theorem}
\label{thm-leg}
Let $n, \, p$ be positive integers. Then 
\begin{equation}
\sum_{k=0}^{n} \binom{2kp}{kp} \binom{2(n-k)p}{(n-k)p} = 
\frac{2^{2np}}{p} \sum_{\ell = 0}^{p-1} e^{i \pi \ell n} 
P_{np} \left( \cos \left( \frac{\pi \ell}{p} \right) \right).
\end{equation}
\end{theorem}
\begin{proof}
Replace the random variable $X_{1} - X_{2}$ considered in the previous section, 
by $X_{1} + WX_{2}$, where 
$W$ is a complex random variable with uniform distribution among the 
$p$-th roots of unity. That is, 
\begin{equation}
\text{Pr} \left\{ W = \omega^{\ell} \right\} = \frac{1}{p}, \quad 
\text{ for } 0 \leq \ell \leq p-1.
\end{equation}
\noindent
The identity \eqref{reduc-2} gives 
\begin{equation}
\mathbb{E} W^{r} = \begin{cases}
1 & \text{ if } r \equiv 0 \pmod p; \\
0 & \text{ otherwise.}
\end{cases}
\end{equation}
\noindent
This is the cancellation alluded above.

Now proceed as in the previous section to obtain the moments
\begin{eqnarray}
\mathbb{E}(X_{1} + W X_{2})^{np}   & = & \sum_{k=0}^{n} \binom{np}{kp} 
\mathbb{E} X_{1}^{(n-k)p} \, \mathbb{E} X_{2}^{kp} \label{sum-3}\\
 & = & \frac{(np)!}{2^{2np}} 
\sum_{k=0}^{n} \binom{2kp}{kp} \binom{2(n-k)p}{(n-k)p}. \nonumber 
\end{eqnarray}

A second expression for $\mathbb{E}(X_{1} + W X_{2})^{np}$ employs 
an alternative form of 
the Legendre polynomial $P_{n}(x)$ defined in \eqref{legen-def}.

\begin{prop}
\label{legen-1}
The Legendre polynomial is given by
\begin{equation}
P_{n}(x) = \frac{1}{n!} \mathbb{E} \left[ (x + \sqrt{x^{2}-1}) X_{1} + 
(x - \sqrt{x^{2}-1}) X_{2} \right]^{n},
\end{equation}
where $X_{1}$ and $X_{2}$ are independent 
$\Gamma\left(\tfrac{1}{2},1\right)$ random 
variables.
\end{prop}
\begin{proof}
The proof is based on characteristic functions. Compute  the sum 
\begin{multline}
\label{charac-1}
\mathbb{E} e^{t (x + \sqrt{x^{2}-1}) \, X_{1}} \, 
\mathbb{E} e^{t (x - \sqrt{x^{2}-1}) \, X_{2}} = \\
\sum_{k=0}^{\infty} \frac{t^{n}}{n!} \, 
\mathbb{E} \left[ (x + \sqrt{x^{2}-1} ) \, X_{1} + (x - \sqrt{x^{2}-1}) 
\, X_{2} \right].
\end{multline}
The moment generating function for a $\Gamma \left( \tfrac{1}{2}, 1 
\right)$ random variable is 
\begin{equation}
\mathbb{E} e^{t X} = ( 1 - t)^{-1/2}.
\end{equation}
\noindent
This reduces \eqref{charac-1} to
\begin{equation*} 
\left( 1 - t ( x + \sqrt{x^{2}-1}) \right)^{-1/2} 
\left( 1 - t ( x - \sqrt{x^{2}-1}) \right)^{-1/2}  = 
(1 - 2tx + t^{2})^{-1/2}
\end{equation*}
\noindent
which is the generating function of the Legendre polynomials.
\end{proof}

\smallskip

\noindent
This concludes the proof of Theorem \ref{thm-leg}.
\end{proof}

\begin{corollary}
Let $x$ be a variable and $\Gamma_{1}, \, \Gamma_{2}$ as before. Then 
\begin{equation}
\mathbb{E} (\Gamma_{1} + x^{2} \Gamma_{2})^{n} = n!x^{n} 
P_{n} \left( \tfrac{1}{2}(x + x^{-1} \right).
\label{jou-1}
\end{equation}
\end{corollary}
\begin{proof}
This result follows from Proposition \ref{legen-1} and the 
change of variables $x \mapsto \tfrac{1}{2}(x+x^{-1})$, known as
the Joukowsky transform. 
\end{proof}

Replacing $x$ by $W^{1/2}$ in \eqref{jou-1} and 
averaging over the values of $W$ gives the second expression for 
$\mathbb{E}( X_{1} + W X_{2})^{np}$. The 
proof of Theorem \ref{thm-leg} is complete.

\section{Chu-vandermonde} 
\label{S:chu} 

The arguments presented here to prove \eqref{mult-identity-2} can be generalized
by replacing the random variables $\Gamma \left( \tfrac{1}{2},1 \right)$ by 
two random variables $\Gamma(a_{i},1)$ with shape parameters $a_{1}$ and 
$a_{2}$, respectively. The resulting identity is the Chu-Vandermonde theorem.

\begin{theorem}
Let $a_{1}$ and $a_{2}$ be positive real numbers. Then 
\begin{equation}
\sum_{k=0}^{n} \frac{(a_{1})_{k}}{k!} \, \frac{(a_{2})_{n-k}}{(n-k)!} = 
\frac{(a_{1}+a_{2})_{n}}{n!}.
\end{equation}
\end{theorem}

The reader will find in \cite{andrews3} a more traditional proof. The paper 
\cite{zeilberger-1995} describes how to find and prove this identity in 
automatic form. 

\medskip

Exactly the same argument for \eqref{mult-identity} provides a multivariable 
generalization of the Chu-Vandermonde identity.

\begin{theorem}
Let $\{ a_{i} \}_{1 \leq i \leq m}$ be a collection of $m$ positive real 
numbers. Then 
\begin{equation}
\sum_{k_{1}+\cdots+k_{m} = n} 
\frac{(a_{1})_{k_{1}}}{k_{1}!} 
\cdots 
\frac{(a_{m})_{k_{m}}}{k_{m}!}  = \frac{1}{n!} 
(a_{1} + \cdots + a_{m})_{n}.
\end{equation}
\end{theorem}

The final stated result presents a generalization of Theorem \ref{thm-leg}.

\begin{theorem}
Let $n, \, p \in \mathbb{N}, a \in \mathbb{R}^{+}$ and $\omega = 
e^{i \pi/p}$. Then 
\begin{equation}
\label{gegen-1}
\sum_{k=0}^{n} \frac{(a)_{kp}}{(kp)!} \, \frac{(a)_{(n-k)p}}{((n-k)p)!} 
z^{2kp} = 
\frac{1}{p} \sum_{\ell=0}^{p-1} e^{i \pi \ell n} z^{np} 
C_{np}^{(a)} \left( \tfrac{1}{2}( z \omega^{\ell} + z^{-1} \omega^{-\ell} )
\right).
\end{equation}
\noindent
Here $C_{n}^{(a)}(x)$ is the Gegenbauer polynomial of degree $n$ and 
parameter $a$.
\end{theorem}
\begin{proof}
Start with the moment representation for the Gegenbauer polynomials
\begin{equation}
\label{mom-geg}
C_{n}^{(a)}(x) = \frac{1}{n!} \mathbb{E}_{U,V} 
\left( U (x+\sqrt{x^{2}-1}) + V ( x - \sqrt{x^{2}-1}) \right)^{n}
\end{equation}
\noindent
with $U$ and $V$ independent $\Gamma(a,1)$ random variables. This
representation is proved in the same way as the proof for the 
Legendre polynomial, replacing the exponent $-1/2$ by and exponent $-a$. Note 
that the Legendre polynomials are Gegenbauer polynomials with 
parameter $a = \tfrac{1}{2}$. This result can also be found in Theorem 3 of 
\cite{sun-p-2007a}.
\end{proof}

\begin{note}
The value $z=1$ in \eqref{gegen-1} gives
\begin{equation}
\sum_{k=0}^{n} \frac{(a)_{kp}}{(kp)!} \, \frac{(a)_{(n-k)p}}{((n-k)p)!} =
\frac{1}{p} \sum_{\ell=0}^{p-1} e^{i \pi \ell n} 
C_{np}^{(a)} \left( \cos \left( \frac{\pi \ell}{p} \right) \right).
\end{equation}
\noindent 
This is a generalization of Chu-Vandermonde.
\end{note}

The techniques presented here may be extended to a variety of situations. 
Two examples illustrate the type of identities that may be proven. They 
involve the Hermite polynomials defined by 
\begin{equation}
H_{n}(x) = (-1)^{n} e^{x^{2}} \left( \frac{d}{dx} \right)^{n} 
e^{-x^{2}}.
\end{equation}

\begin{theorem}
Let $m \in \mathbb{N}$. The Hermite polynomials satisfy 
\begin{equation}
\label{multiH}
\frac{1}{n!} H_{n} \left( \frac{x_{1} + \cdots + x_{m}}{\sqrt{m}} \right) 
= m^{-n/2} \sum_{k_{1}+\cdots + k_{m}}
\frac{H_{k_{1}}(x_{1})}{k_{1}!}
\cdots 
\frac{H_{k_{m}}(x_{m})}{k_{m}!}.
\end{equation}
\end{theorem}
\begin{proof}
Start with the moment representation for the Hermite polynomials 
\begin{equation}
H_{n}(x) = 2^{n} \mathbb{E}(x + i N)^{n},
\end{equation}
\noindent
where $N$ is normal with mean $0$ and variance $\tfrac{1}{2}$. The details 
are left to the reader.
\end{proof}

The moment representation for the Gegenbauer polynomials \eqref{mom-geg}
yields the final result presented here.

\begin{theorem}
Let $m \in \mathbb{N}$. The Gegenbauer polynomials $C_{n}^{(a)}(x)$ 
satisfy
\begin{equation}
\label{multiG}
C_{n}^{(a_{1}+\cdots+a_{m})}(x) = 
\sum_{k_{1}+\cdots + k_{m}=n} 
C_{k_{1}}^{(a_{1})}(x)
\cdots 
C_{k_{m}}^{(a_{m})}(x).
\end{equation}
\end{theorem}

\begin{remark}
A relation between Gegenbauer and Hermite polynomials is given by 
\begin{equation}
\lim\limits_{a \to \infty} \frac{1}{a^{n/2}} C_{n}^{(a)} 
\left( \frac{x}{\sqrt{a}} \right) = \frac{1}{n!} H_{n}(x).
\end{equation}
This relation allows to recover easily identity \eqref{multiH} 
from identity \eqref{multiG}.
\end{remark}

\medskip

The examples presented here, show that many of the classical identities 
for special functions may be established by probabilistic methods.  The 
reader is encouraged to try this method in his/her favorite identity. 

\bigskip

\noindent
\textbf{Acknowledgements}. The work of the second author was 
partially supported by NSF-DMS 0070567.


\end{document}